\documentclass{article}

\usepackage[english]{babel}
\usepackage{amssymb}
\usepackage{amsmath}
\usepackage{amsthm}
\usepackage{mathtools}
\usepackage[linesnumbered,ruled,vlined]{algorithm2e} 

\usepackage[letterpaper,top=2cm,bottom=2cm,left=3cm,right=3cm,marginparwidth=1.75cm]{geometry}

\usepackage{graphicx}
\usepackage{float}
\newtheorem{definition}{Definition}
\newtheorem{theorem}{Theorem}

\newtheorem{lemma}{Lemma}
\newtheorem{example}{Example}
\providecommand{\keywords}[1]{\textbf{\textit{Keywords---}} #1}
\usepackage[colorlinks=true, allcolors=blue]{hyperref}
\author{RaGon Ebker\\ email \href{ragon@disroot.org}{raphaelragon.ebker@studium.fernuni-hagen.de}} 

\title{Solving non-separable polynomials over the field of Puiseux series via golden lifting}

\begin{document}
\maketitle

\begin{abstract}\label{abstract}
We develop an iterative method to calculate the roots of arbitrary polynomials over the field of Puiseux series including non-separable ones. The method works by transforming the polynomial and its roots into a special form and then extracting a new, univariate polynomial that contains information about our roots. We also provide a working implementation of the algorithm in Python.
\end{abstract}
\section{Notation}\label{notation}
Let $\mathbb{N}$ be the set of natural numbers including $0$.
Let $K$ be a field and $K((x^\frac{1}{n}))$ be the field of Puiseux series over $K$. Let the elements of $K((x^\frac{1}{n}))$ have the form $y = \sum_{k=k_0} ^\infty b_{k}x ^{\frac{k}{n}}, n \in \mathbb{Z}$. When $y$ has just finitely many terms, we call $d_x$ the degree of $y$. Let $Q : K((x)) \mapsto K((x))$ be a polynomial over the field of Puiseux series, $Q(y) = a_{d_y} y^{d_y} + ... + a_1 y + a_0, d_y \in \mathbb{N^+}$, $d_y$ being the degree of $Q$.
Let $\alpha =  \sum_{k=0} ^\infty b_{k}x ^{\frac{k}{n}}, n \in \mathbb{Z}$ be a root of $Q$.
\begin{definition}[s-multiplicity]
A polynomial has $s-$multiplicty, when there exist $s$ roots $\alpha_1,...,\alpha_s$ of $Q$ with the coefficients $b_0$ all being $0$.
\end{definition}
\begin{definition}[s-plus-multiplicity]
A polynomial $Q$ has $s^+-$Multiplicty, when there exist $s^+$ roots $\alpha_1,...,\alpha_{s^+}$ of $Q$, $s^+ \in \{1,...,s\}, s$ being the $s-$ Multiplicity of Q with the coefficients $b_0$ all being $0$ and the term $b_1$ has the same valuation for all $\alpha_j, j \in \{1,...,s^+\}$.
\end{definition}
\begin{example}
    The polynomial $Q(y) = (y-(1+x+x^2))(y-x^{0.5})(y-x^{0.6})(y-x^{0.5} + x^2)$ has $s$-Multiplicty 3 and $s^+$-Multiplicty 2.
\end{example}
Let $v: K((x)) \mapsto \mathbb{Q}$ be the valuation map defined by $v(y) = k_0$. 
\section{Introduction}\label{introduction}
Let's imagine for a moment, that we have an algorithm to calculate the smallest root of a polynomial $Q$, for example, $Q(y) = (y - (1+x+x^2)) (y-(2+x+x^2))$ over the field of Puiseux Series. In this case, the smallest root is $(1+x+x^2)$. Small means the root with the leading coefficient that has the smallest valuation, in this case. But we can only calculate its term with the highest evaluation, i.e. $1$. 
\newline
How can we proceed to calculate the next terms of the root? The answer may seem obvious: By shifting the roots of the polynomial $Q$. Now we have the shifted Polynomial $$Q_{shift} = (y - (x+x^2)) (y-(1x+x^2))$$ and we can easily calculate the next term with our algorithm: $x$. 
In this paper we are going to explore this idea, mainly:
\newline
1. How to develop an algorithm to calculate the smallest root under certain conditions
\newline
2. How to transform and shift our polynomial to fulfill this condition
\medskip
\newline
A good overview of this process can also be seen in section \ref{algorithm}.
Comparable methods to solve this problem over the Puiseux series and power series exist for example in the form of Hensels Lemma and its developed versions \cite{neiger} or in the Newton Puiseux method \cite{brieskorn2012plane}. 
\section{Main Result}\label{main}
In this section, we show the main result. It works by reducing our polynomial to a smaller one, under certain requirements. In section 4 we are going to see how we can transform every polynomial into one that satisfies these exact requirements. We calculate the root of the polynomial coefficient by coefficient just as in the original Newton-Puiseux algorithm.
\begin{theorem}[Golden Lifting]
Let $Q : K((x)) \mapsto K((x))$ be a polynomial over the field of Puiseux series, $Q(y) = a_{d_y} y^{d_y} + ... + a_1 y + a_0, d_y \in \mathbb{N^+}$. Let $\alpha_1,...,\alpha_n$ be the roots of our polynomial and let $v(a_i) \geq 0, i \in \{1,...,d_y\}$. Now we assume $v(\alpha_j) \geq e > 0 \  , j \in \{1,...,s\}$, $s \in \{1,...,d_y-1\} , e \in \mathbb{Q}$, with $v(\alpha_j) = e$ for $j \in \{1,...,s^+\}, s^+ >1$.  $s$ is exactly the $s-$multiplicity of $\alpha_j$, $s^+$ the $s^+$-Multiplicity. $e$ is the smallest valuation of all $\alpha_j$.  We can further represent those roots as $\alpha_j = c_{1_j} x^{e} + c_{2_j} x^{e + \gamma_{2_j}} + c_{3_j} x^{e + \gamma_{2_j} + \gamma_{3_j}} ..., \gamma_j \in \mathbb{Q}, j \in \{1,...,s^+\}.$ 
\newline
Now the aim of this theorem is to calculate the $c_{1_j}, j \in \{1,...,s^+\} $.
\newline
We further assume $v(\alpha_j) = 0, \ \forall j \in \{s+1,...,d_y\}.$We now have a look at the coefficients of $Q$. We remember  
$$a_i = b_{i_0}x^{\delta_1} + b_{i_1}x^{\delta 1 + \delta_2} + ..., \ \delta_k \in \mathbb{Q}, k \in \mathbb{N^+} .$$ 
\bigskip
\newline
Let now $Q_R(x) = b_{s_{min}}x^{d_y} + ... + b_{1_{min}} x$  be such that $b_{i_{min}} := \min_{l \in \{1,...,d_x\}} \{b_{i_l} \neq 0 \}, i \in \{1,...,s\}$.
\medskip
\newline
Then the $s$ roots of $Q_R$ are exactly the $c_{1_j}$.
Then $Q_R$ is of degree $s$, and $c^+$ roots are exactly the $c_{1_j} j \{1,...,s^+\}$ belonging to the the roots with the lowest valuation. The other $s-s^+$ roots are zero.
\end{theorem}

\begin{proof}
We know that $c_{1_j}, j \in \{1,...,s^+\}$ being the correct coefficient of a root of $Q$ is equivalent to $Q(c_{1_j} x^{e}) \mod  x^{e'} \equiv 0$, i.e. $c_{1_j}x$ is a root of $Q(y) \mod  x^{e'}$. $e' \in \mathbb{Q}$ is in this case any number bigger than $s \cdot e$ and smaller than the exponent of any term of $Q$, that has a valuation bigger than $e$.  This is can also be seen by looking at the polynomial in its linear factorization.
We show that $Q(\beta x^e) = Q_R(\beta x^e) \mod x^{e'}$:
\begin{align*}
     Q(\beta x^e) &\equiv  a_{d_y} (\beta x^e)^{d_y} + ... + a_1 \beta x^e + a_0 \mod x^{e'} \\
         &\equiv a_{d_y} ^* (\beta x^e)^{d_y} + ... + a_1^* (\beta x^e) + a_0^* + \left( b_{d_{y_{min}}}(\beta x^e) ^{d_y} + ... + b_{1_{min}}(\beta x^e) + b_{0_{min}} \right) \mod x^{e'} \\
         &\equiv b_{d_{y_{min}}}(\beta x^e)^{d_y} + ... + b_{1_{min}}(\beta x^e) + b_{0_{min}} \mod x^{e'} \\
         &\equiv b_{s_{min}}(\beta x^e)^s + ... + b_{1_{min}}(\beta x^e) + b_{0_{min}} \mod x^{e'} \\
\end{align*}
For $a_i^{*} = a_i - b_{i_{min}}, i \in \{1,...,d_y\}$.
\end{proof}
So to conclude: When we solve $Q_R$ we obtain the solutions $c_{1_j}x^e$. We will explain how to solve $Q_R$ with the help of shifts in section 4.
\medskip
To apply this theorem, we need to fulfill the condition $v(\alpha_j) \geq e > 0, j \in \{1,...,s\}$ with $v(\alpha_j) = e$ for at least one $j \in \{1,...,s\}$ and $v(\alpha_k) = 0, \ \forall k \in \{s+1,...,d_y\}$ (1).
We also need to obtain $e$ (2) and $s$ (3).   Once we have applied this step, we can extract (1), (2), and (3), for the next step, from the set of roots of $Q_R$.
We can proceed equally to encounter the next coefficients $c_{2_k},...,c_{d_{y_l}}$, after transforming our polynomial to fulfill our assumption again. $k$ and $l$ are the corresponding indices.
\section{Initial shift of the Polynomial}
In this section, we are going to prove how to transform a polynomial $Q: K((x^\frac{1}{n})) \mapsto K((x^\frac{1}{n}))$ into one that fulfills the condition (1), if it does not already, and how to extract (2) and (3). 
We first start with a commonly know Lemma:
\begin{lemma}
Suppose we have a polynomial $P$ over the $K$, with $\alpha_1,...,\alpha_n, n \in N^+$ being its roots. Then for any constant $c \in K$ the polynomial with roots $\alpha_1 + c,...,\alpha_n +c $ has the form
    $$Q(y)=P(y-c)=a_{n}(y-c)^{n}+a_{n-1}(y-c)^{{n-1}}+\cdots +a_{{0}}.$$
    and the polynomial with the roots $c\alpha_1,...,c\alpha_n$ has the form $$Q(y) = a_{n}y^{n}+a_{n-1}cy^{{n-1}}+\cdots +a_{{0}}c^{n}.$$
\end{lemma}
\medskip
We now start with (1).
We can check (1) by calculating the constant parts of the roots via $Q_{|x=0}$. If we have at least one root of $Q_{|x=0}$ which is unequal to zero and one that is equal to zero we fulfill the conditions. Checking $Q_{|x=0}$, mainly to get the constant part of a root of $Q$, is a common technique and can be seen by looking at $Q$ in linear factor representation. 
\medskip
\newline
CASE 1 
\newline
First we want to ensure that $v(\alpha_j) = 0$ holds for at least one $i \in \{1,...,d_y\}$.  What if all roots of  $Q_{|x=0}$ are zero for all $j \in \{1,...,d_y\}$? In this case, we shift our polynomials via multiplication as in Lemma 1.  For this end, we actually need to know $e$, the valuation of $\alpha_j$, which we can obtain via the Newton Polygon of our polynomial. We thus multiply our roots with $a^e$.
\medskip
\newline
Going back to our now eventually shifted polynomial, which has $v(\alpha_j) = 0$ holds for at least one $j \in \{1,...,d_y\}$.
\newline
\medskip
\newline
CASE 2
\newline
Now we check if $v(\alpha_j) > 0$ for at least one $j \in \{1,...,d_y\}$ Suppose we have $v(\alpha_j) = 0, \forall i \in \{1,...,n\}$, then we can get the constant term of the $\alpha_j$ by evaluating $Q_{|x=0}$ and taking its roots, as we have already discussed. Once we have obtained all the constant parts $c_j$ of $\alpha_j$ we are going to use them to shift our polynomial with help of Lemma 1.
If we end up with a polynomial with $v(\alpha_j) > 0$ for all $j \in \{1,...,d_y\}$, we go back to case 1. If not we have one with the desired condition.
\medskip
\newline
Now we can finally talk about the case when (1) is fulfilled:
We calculate the roots of $Q_{|x=0}$ or take the already calculated roots, and one of them is now unequal to zero. Let's call it $\alpha_k, k \in \{1,...,d_y\}.$ In this step we already obtain the $s$ -multiplicity of that root, which is exactly the multiplicity we need for our next step. After this calculation, we shift again, and so on, reaching our root iteratively.
\newline
\medskip
\section{Algorithm}\label{algorithm}
In this section, we are going to explore the algorithm in detail. An implementation can be found at \cite{ragon}. First, we have a look at the flow diagram in figure
\ref{fig:flowfigure}.
\begin{figure}\label{flowfigure}
\centering
\includegraphics[width=1\linewidth]{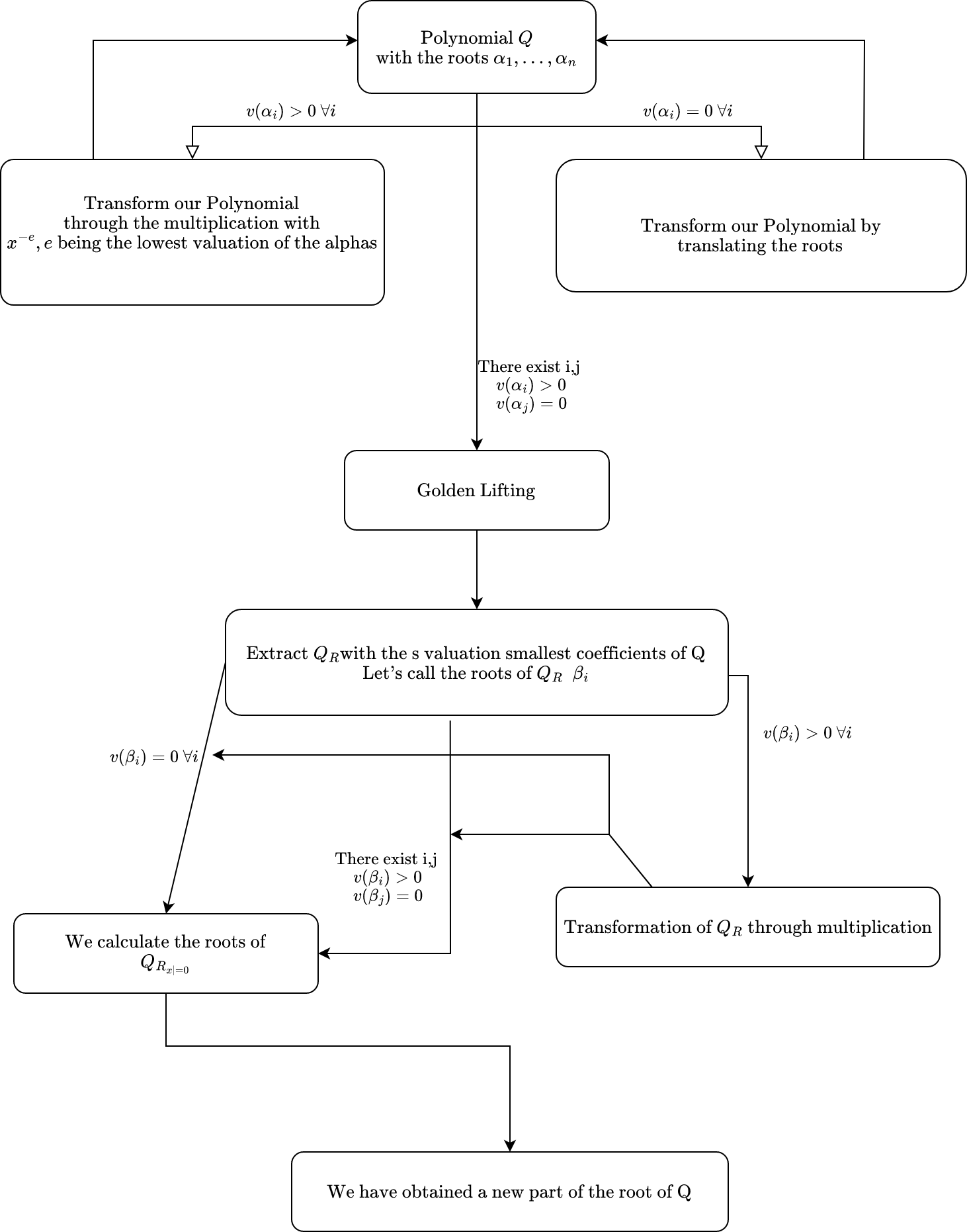}
\caption{\label{fig:flowfigure}Flowchart of the algorithm}
\end{figure}
\newline
\begin{algorithm}[H]
\caption{info class}\label{alg:class}
def info\_constructor(self,d,root\_dict\_list,x\_lift): \\
\ \   self.d = d \\
 \ \  self.root\_dict\_list = root\_dict\_list \\ 
\   \ self.x\_lift = x\_lift \\
\end{algorithm}
For the implementation, we first define an info class that saves all of our important information. When initializing the algorithm, we create an object from the info class, with specified precision $d$, an empty root\_dict\_list, and an x\_lift of $0$. The current alpha can always be calculated by combining the parts of our root from root\_dict\_list. We then give the info object to the method calculate\_smallest\_root, our main method:
\newline
\begin{algorithm}[H]
\caption{calculate\_smallest\_root}\label{alg:calculate}

\SetKwInOut{Input}{Input}
\SetKwInOut{Output}{Output}
\Input{Polynomial $p$, info object }
\Output{The info object that contains, among other information, the root alpha of $p$ with the lowest valuation}
p\_{shift} $\gets$ calculate\_initial\_shift(p,info) \;
\uIf{p\_shift(info.alpha) == 0}{
    return
  }
\For{$i \gets 1$ \KwTo info.d}{
    info.d $\gets$ info.d -1 \;
    p\_shift = shift\_horizontally(p\_shift,calculate\_h\_shift(info)) \;
    multiplicity = info.last\_root.multiplicity \;
    golden\_lifting(p\_shift,multiplicity,info) \;
    \uIf {p\_shift(info.alpha) == 0}{
    return 
    }
    }
    return
\end{algorithm}
As we can see in \ref{alg:calculate}, we follow the process described in sections 3 and 4. We first calculate the initial shift, to bring our polynomial into the form $v(\alpha_j) \geq e > 0, j \in \{1,...,s\}$ with $v(\alpha_j) = e$ for at least one $j \in \{1,...,s\}$ and $v(\alpha_k) = 0, \ \forall k \in \{s+1,...,d_y\}$, as described in (1).
In this process, we already obtain the first part of our first root alpha. This also brings us the $s-$ multiplicity of the next part of the root.
In the for loop, starting on line 4, we start shifting our polynomial horizontally, so it attains the form (1) again. Shifting a polynomial horizontally means adding a constant to its root.
\begin{algorithm}[H]
\caption{calculate\_initial\_shift}\label{alg:calculate_i}
\SetKwInOut{Input}{Input}
\SetKwInOut{Output}{Output}
\Input{Polynomial $p$, info object}
\Output{Shifted polynomial $p$ or p\_shift}
root\_dict, shift\_number = get\_sub\_x\_root(p,info) \\
\uIf {shift\_number == 0}{
    info.root\_dict\_list.append(root\_dict) \\
    return $p$
    }
\Else {
      slopes $\gets$ get\_newton\_slopes(p) \\
      min\_slope $\gets$ min\_not\_zero(slopes) \\
      p\_shift $\gets$ shift\_vertically(p,x**(-min\_slope))  \\
      info.x\_lift $\gets$ min\_slope \\
      info.d $\gets$ info.d-1 \\
      info  $\gets$ calculate\_smallest\_root\_q\_x(p\_shift,info) \\
      return p\_shift
}
\end{algorithm}
The initial shift method calculates the roots of $p_{|x=0}$ in line 1 with get\_sub\_x\_root. Then we shift our polynomial vertically if all of the roots, in this case described by shift\_number, are unequal to zero. Vertical shifting means the multiplication of the roots with a constant.
\newline
\begin{algorithm}[H]
\caption{calculate\_smallest\_root\_q\_x}\label{alg:calculate_q}
\SetKwInOut{Input}{Input}
\SetKwInOut{Output}{Output}
\Input{Polynomial $p$, info object}
\Output{Shifted polynomial $p$ or p\_shift}
root\_dict, shift\_number = get\_sub\_x\_root(p,info) \\
\uIf {shift\_number == 0}{
    info.root\_dict\_list.append(root\_dict) \\
    return
    }
\Else {
    slopes $\gets$ get\_newton\_slopes(p)
    
    min\_slope $\gets$ min\_not\_zero(slopes) \\
    p\_shift $\gets$ shift\_vertically(p,x**(-min\_slope))  \\
    info.x\_lift $\gets$ min\_slope \\
    info.d $\gets$ info.d-1 \\
    info  $\gets$ calculate\_smallest\_root\_q\_x(p\_shift,info) \\
    return
}
\end{algorithm}
Algorithm \ref{alg:calculate_q} is the same as \ref{alg:calculate_i} but it does not return anything. 
\newline
\begin{algorithm}[H]
\caption{golden\_lifting}\label{alg:golden}
\SetKwInOut{Input}{Input}
\SetKwInOut{Output}{Output}
\Input{p\_shift,multiplicity,info}
    shifted\_coeffs $\gets$ reversed(p\_shift.coeffs()) \\
    cutoff\_coeffs $\gets$ [] \\
    \For{$i \gets 1$ \KwTo info.d}{
    coeff $\gets$ shifted\_coeffs[i] \\
    cutoff\_coeffs.append(term\_with\_lowest\_valuation(coeff))
    }
    new\_poly $\gets$ Monic(Poly(reversed(cutoff\_coeffs))) \\
    \uIf {multiplicity $>$ 1}{
        calculate\_smallest\_root\_q\_x(new\_poly,info) \\
        return
    }
    \Else {
    r $\gets$ roots(new\_poly) \\
    info.root\_dict\_list.append(order\_roots(r))
    return
    }
\end{algorithm}
The algorithm \ref{alg:golden} describes exactly the process of extracting $p_R$ from our polynomial $p$, just as we extracted $Q_R$ from $Q$ in theorem 1. Here it is also possible to check if the polynomial $p_R$, which is called new\_poly in the code and pseudocode, has a very simple form for example when it consists only of a single root with multiplicity $s$. This would mean that its $s$-multiplicity and $s^+$ multiplicity are the same. We then start the same process as in the initial\_shift method to calculate the roots of $p_R$ i.e. new\_poly. 

\end{document}